\documentclass[12pt]{amsart}
\usepackage[margin=1in]{geometry}

\newtheorem{theorem}{Theorem}[section]
\newtheorem{lemma}[theorem]{Lemma}
\newtheorem{conjecture}[theorem]{Conjecture}

\newtheorem{corollary}[theorem]{Corollary}

\theoremstyle{definition}

\newtheorem{xca}[theorem]{Exercise}

\theoremstyle{remark}
\newtheorem{remark}[theorem]{Remark}

\begin{document}

\title[]
{Supercongruences for\\
the Almkvist-Zudilin numbers}



\author{Tewodros Amdeberhan and Roberto Tauraso}
\address{Department of Mathematics,
Tulane University, New Orleans, LA 70118}
\email{tamdeber@tulane.edu}

\address{Dipartimento di Matematica,
 Universita' di Roma "Tor Vergata",
 Via della Ricerca Scientifica, 1, 
 00133 Roma  -   Italy}
\email{tauraso@mat.uniroma2.it}

\subjclass[2010]{Primary ??}

\date{\today}


\begin{abstract}
Given a prime number $p$, the study of divisibility properties of a sequence $c(n)$ has two contending approaches: $p$-adic valuations and superconcongruences. The former searches for the highest power of $p$ dividing $c(n)$, for each $n$; while the latter (essentially) focuses on the maximal powers $r$ and $t$ such that $c(p^rn)$ is congruent to $c(p^{r-1}n)$ modulo $p^t$. This is called supercongruence. In this paper, we prove a conjecture on supercongruences for sequences that have come to be known as the Almkvist-Zudilin numbers. Some other (naturally) related family of sequences will be considered in a similar vain.
\end{abstract}

\maketitle

\newcommand{\ba}{\begin{eqnarray}}
\newcommand{\ea}{\end{eqnarray}}
\newcommand{\ift}{\int_{0}^{\infty}}
\newcommand{\nn}{\nonumber}
\newcommand{\no}{\noindent}
\newcommand{\lf}{\left\lfloor}
\newcommand{\rf}{\right\rfloor}
\newcommand{\realpart}{\mathop{\rm Re}\nolimits}
\newcommand{\imagpart}{\mathop{\rm Im}\nolimits}

\newcommand{\op}[1]{\ensuremath{\operatorname{#1}}}
\newcommand{\pFq}[5]{\ensuremath{{}_{#1}F_{#2} \left( \genfrac{}{}{0pt}{}{#3}
{#4} \bigg| {#5} \right)}}

\newtheorem{Definition}{\bf Definition}[section]
\newtheorem{Thm}[Definition]{\bf Theorem}
\newtheorem{Example}[Definition]{\bf Example}
\newtheorem{Lem}[Definition]{\bf Lemma}
\newtheorem{Cor}[Definition]{\bf Corollary}
\newtheorem{Prop}[Definition]{\bf Proposition}
\numberwithin{equation}{section}

\section{Introduction}

\noindent
The {\em Ap\'ery numbers} $A(n)=\sum_{k=0}^n\binom{n}k^2\binom{n+k}k^2$ were valuable to R. Ap\'ery in his celebrated proof ~\cite{A} 
that $\zeta(3)$ is an irrational number. Since then these numbers have been a subject of much research. For example, they stand among a host of other sequences with the property 
$$A(p^rn)\equiv_{p^{3r}} A(p^{r-1}n)$$
now known as {\em supercongruence} $-$ a term dubbed by F. Beukers ~\cite{B}.

\smallskip
\noindent
At the heart of many of these congruences sits the classical example $\binom{pb}{pc}\equiv_{p^3}\binom{b}c$ which is a stronger variant of the famous congruence $\binom{pb}{pc}\equiv_p\binom{b}c$ of Lucas. For a compendium of references on the subject of Ap\'ery-type sequences, see ~\cite{S}.

\smallskip
\noindent
Let us begin by fixing  notational conventions. Denote the set of positive integers by $\mathbb{N}^+$. For $m\in\mathbb{N}^+$, let $\equiv_m$ represent congruence modulo $m$. Throughout, assume $p\geq5$ is a prime.

\smallskip
\noindent
In this paper, true to tradition, we aim to investigate similar type of supercongruences for the following family of sequences. For integers $i\geq0$ and $n\geq1$, define
\begin{align*} a_i(n):&=\sum_{k=0}^{\lfloor(n-i)/3\rfloor}(-1)^{n-k}\binom{3k+i}k\binom{2k+i}k\binom{n}{3k+i}\binom{n+k}k3^{n-3k-i} \\
\end{align*}
In recent literature, $a_0(n)$ are referred to as the Almkvist-Zudilin numbers.
Our motivation for the present work here emanates from the following claim found in ~\cite{OS} (see also ~\cite{CCS}, ~\cite{OSS}).

\begin{conjecture}\label{MainC}
For a prime $p$ and $n\in\mathbb{N}^+$, the Almkvist-Zudilin numbers satisfy
$$a_0(pn)\equiv_{p^3}a_0(n).$$
\end{conjecture}

\smallskip
\noindent
Our main results can be summed up as: 

{\em if $p$ is a prime and $n\in\mathbb{N}^+$, then $a_0(pn)\equiv_{p^3}a_0(n)$ and $a_i(pn)\equiv_{p^2}0$ for $i>0$.}

\smallskip
\noindent
The organization of the paper is as follows. Section 2 lays down some preparatory results to show the vanishing of $a_i(pn)$ modulo $p^2$, for $i>0$. Section 3 sees the completion of the proof. Our principal approach in proving the main conjecture $a_0(pn)\equiv_{p^3}a_0(n)$ relies on a ``machinery'' we develop as a proof strategy which maybe described schematically as:
$$\text{{\em reduction}} + \text{{\em $p$-identities}}.$$
Sections 4 and 5 exhibit its elaborate execution. The reduction brings in a {\em tighter} claim and it also offers an advantage in allowing to work with a single sum instead of a double sum. In Section 6, we complete the proof for Conjecture \ref{MainC}. The paper concludes with Section 7 where we declare an improvement on the results from Section 3 which states a congruence for the family of sequences $a_i(pn)$ modulo $p^3$, when $i>0$. Furthermore, in this last section, the reader will find a proof outline guided by our ``machinery''.

\section{Preliminary results}

\noindent
{\em Fermat quotients} are numbers of the form $q_p(x)=\frac{x^{p-1}-1}p$ and they played a useful role in the study of cyclotomic fields and Fermat's Last Theorem, see ~\cite{R}. 
The next three lemmas are known and we give their proofs for the sake of completeness.
\begin{lemma}
If $a \not \equiv_p 0$ then for $d\in \mathbb{Z}$,
\begin{align}\label{FT}
q_p(a^d)\equiv_{p^2}d\,q_p(a)+p\binom{d}{2}\,q_p(a)^2.
\end{align}
\end{lemma}
\begin{proof} Since by Fermat's little theorem $a^{p-1}\equiv_p 1$ then it follows that
$$\left(a^{p-1}\right)^d=\left(1+(a^{p-1}-1)\right)^d
\equiv_{p^3}1+
d(a^{p-1}-1)+\binom{d}{2}(a^{p-1}-1)^2.$$
\end{proof}

\begin{lemma} Let $H_n=\sum_{j=1}^n \frac{1}{j}$ be the $n$-th harmonic number. Then, for $n\in\mathbb{N}^+$, we have 
\begin{equation}\label{IdH}
\sum_{k=1}^n(-1)^k\binom{n}k\binom{n+k}k\frac1k=-2H_n.
\end{equation}
\end{lemma}
\begin{proof} For an indeterminate $y$, a simple partial fraction decomposition proves the identity (see \cite[Lemma 3.1]{M})
\begin{align}\label{Mort}
\sum_{k=0}^n(-1)^k\binom{n}k\binom{n+k}k\frac1{k+y}=\frac{(-1)^n}y\prod_{j=1}^n\frac{y-j}{y+j}.\end{align}
Now, subtract $\frac1y$ from both sides and take the limit as $y\rightarrow0$. The right-hand side takes the form
$$\frac1{n!}\lim_{y\rightarrow0}\left[\frac{\prod_{j=1}^n(j-y)-\prod_{j=1}^n(j+y)}y\right]=-2\sum_{k=1}^n\frac1k.$$
The conclusion is clear.
\end{proof}

\begin{lemma}\label{PC}
Suppose $p$ is a prime and $0\leq k <p/3$. Then,
$$(-1)^k\binom{\lfloor p/3\rfloor}k\binom{\lfloor p/3\rfloor+k}{k}
\equiv_p\binom{3k}{k,k,k}3^{-3k}.$$
\end{lemma}
\begin{proof} We observe that $\binom{n}k\binom{n+k}k=\binom{2k}k\binom{n+k}{2k}$. If $p\equiv_31$, then $\lfloor\frac{p}3\rfloor=\frac{p-1}3$
and hence
\begin{align*} \binom{\frac{p-1}3+k}{2k}
&=\frac{\frac{p-1}3(\frac{p-1}3+k)}{(2k)!}\prod_{j=1}^{k-1}\left(\frac{p-1}3\pm j\right) \\
&\equiv_p\frac{(-1)^k(3k-1)}{3^{2k}(2k)!}\prod_{j=1}^{k-1}(3j\pm1) 
=\frac{(-1)^k(3k)!}{3^{3k}(2k)!k!}. \end{align*}
Therefore, we gather that
$$(-1)^k\binom{\frac{p-1}3}k\binom{\frac{p-1}3+k}k=
(-1)^k\binom{2k}k\binom{\frac{p-1}3+k}{2k}\equiv_p\frac{(3k)!}{3^{3k}!k!^3}=\binom{3k}{k,k,k}3^{-3k}.$$
The case $p\equiv_3-1$ runs analogously. \end{proof}

\begin{corollary}\label{C27} 
For a prime $p$ and an integer $0<i<\frac{p}3$, we have the congruences
\begin{align*} \sum_{k=1}^{p-1}\binom{3k}{k,k,k}\frac{3^{-3k}}k&\equiv_{p}
\sum_{k=1}^{\lfloor p/3\rfloor}\binom{3k}{k,k,k}\frac{3^{-3k}}k\equiv_{p}3q_p(3), \\
\sum_{k=0}^{p-1}\binom{3k}{k,k,k}\frac{3^{-3k}}{k+i}&\equiv_{p}
\sum_{k=0}^{\lfloor p/3\rfloor}\binom{3k}{k,k,k}\frac{3^{-3k}}{k+i}\equiv_{p}0.
\end{align*}
\end{corollary}
\begin{proof} For the first assertion, we combine \eqref{IdH}, Lemma \ref{PC} and
the congruence (\cite[p. 358]{L})
$$H_{\lfloor p/3\rfloor}\equiv_{p}-3\sum_{r=1}^{\lfloor p/3\rfloor}\frac1{p-3r}
\equiv_{p}-\frac{3q_p(3)}2.$$
The second congruence follows from \eqref{Mort} with $y=i$ and Lemma \ref{PC}.
\end{proof}

\section{Main results on the sequences $a_i(n)$ for $i>0$}

\begin{theorem} For a prime $p$ and $n, i\in\mathbb{N}^+$ with $i<\frac{p}3$, we have $a_i(pn)\equiv_{p^2} 0$.
\end{theorem}
\begin{proof} Let $k=pm+r$ for $0\leq r\leq p-1$. Note: $3k+i=3pm+3r+i\leq pn$. Write
\begin{align*} a_i(pn)=\sum_{m=0}^{\lfloor n/3\rfloor}\sum_{r=0}^{p-1}
(-1)^{pn-pm-r}&\binom{3pm+3r+i}{pm+r}\binom{2pm+2r+i}{pm+r}  \\
&\cdot\binom{pn}{3pm+3r+i}\binom{pn+pm+r}{pm+r}3^{pn-3pm-3r-i}.\end{align*}
If $t:=3r+i\geq p+1$, it is easy to show that the following terms vanish modulo $p^2$:
$$\binom{3pm+t}{pm+r}\binom{2pm+2r+i}{pm+r}\binom{pn}{3pm+t}=
\binom{3pm+t}{pm+r,pm+r,pm+r+i}\binom{pn}{3pm+t}.$$
Therefore, we may restrict to the remaining sum with $3r+i\leq p$:
\begin{align*} a_i(pn)=\sum_{m=0}^{\lfloor n/3\rfloor}\sum_{r=0}^{\lfloor(p-i)/3\rfloor}
(-1)^{n-m-r}&\binom{3pm+3r+i}{pm+r}\binom{2pm+2r+i}{pm+r} \\
&\cdot \binom{pn}{3pm+3r+i}\binom{pn+pm+r}{pm+r}3^{pn-3pm-3r-i}.\end{align*}
We need Lucas's congruence $\binom{pb+c}{pd+e}\equiv_p\binom{d}d\binom{c}e$ to arrive at 
\begin{align*} a_i(pn)\equiv_p\sum_{m=0}^{\lfloor n/3\rfloor}\sum_{r=0}^{\lfloor(p-i)/3\rfloor}
(-1)^{n-m-r}&\binom{3m}m\binom{3r+i}r\binom{2m}m\binom{2r+i}r \\
&\cdot \binom{pn}{3pm+3r+i}\binom{n+m}m3^{pn-3pm-3r-1}.\end{align*}
For $0<j<p$, we apply Gessel's congruence $\binom{p}j\equiv_{p^2}(-1)^{j-1}\frac{p}j$ (if $p=3r+i$, in this case, still the corresponding term properly absorbs into the sum below) so that
\begin{align*} \binom{pn}{3pm+3r+i}&=\frac{pn}{3pm+3r+i}\binom{pn-1}{3pm+3r+i-1}
=\frac{pn}{3pm+3r+i}\binom{p(n-1)+p-1}{3pm+3r+i-1} \\
&\equiv_{p^2}(-1)^{r+i-1}\frac{pn}{3r+i}\binom{n-1}{3m}, \end{align*}
which leads to
\begin{align*} a_i(pn)\equiv_{p^2}pn\sum_{m=0}^{\lfloor n/3\rfloor}\sum_{r=0}^{\lfloor(p-i)/3\rfloor}(-1)^{n-m-r}&\binom{3m}m\binom{3r+i}r\binom{2m}m\binom{2r+i}r \\
&\cdot\frac{(-1)^{r+i-1}}{3r+i}\binom{n-1}{3m}\binom{n+m}m3^{pn-3pm-3r-i}.\end{align*}
Next, we use Fermat's Little Theorem and {\em decouple} the double sum to obtain
\begin{align*} a_i(pn)\equiv_{p^2}n\sum_{m=0}^{\lfloor n/3\rfloor}(-1)^{n-m+i-1}
3^{n-3m-i}&\binom{3m}m\binom{2m}m\binom{n-1}{3m}\binom{n+m}m \\
&\cdot p\sum_{r=0}^{\lfloor(p-i)/3\rfloor}\binom{3r+i}r\binom{2r+i}r\frac{3^{-3r}}{3r+i}.\end{align*}
It suffices to verify the {\em sum over $r$} vanishes modulo $p$. To achieve this, apply partial fraction decomposition and Corollary \ref{C27} (upgrading the sum to $\lfloor p/3\rfloor$ is {\em harmless} here). Thus, 
\begin{align*} \sum_{k=0}^{\lfloor p/3\rfloor}\binom{3k+i}{k}\binom{2k+i}{i}\frac{3^{-3k}}{3k+i}&=\sum_{k=0}^{\lfloor p/3\rfloor}\binom{3k}{k,k,k}3^{-3k}\prod_{j=1}^{i-1}(3k+j)\prod_{j=1}^i(k+j)^{-1} \\
&=\sum_{j=1}^i\alpha_j(i)\sum_{k=0}^{\lfloor p/3\rfloor}
\binom{3k}{k,k,k}\frac{3^{-3k}}{k+j}\equiv_p\sum_{j=1}^i\alpha_j(i)\cdot0=0;
\end{align*}
where $\alpha_j(i)\in\mathbb{Q}$ are some constants. We have enough reason to conclude the proof.
\end{proof}

\section{The reduction on the sequence $a_0(n)$}

\noindent
Our proof of Conjecture \ref{MainC} requires a slightly more delicate analysis than what has been demonstrated in the previous sections for the sequences $a_i(n)$, where $i>0$. As a first major step forward, we state and prove the following \emph{somewhat} stronger result. This will be crucial in scaling down a double sum, which emerges (see proof below) as an expression for the sequence $a_0(pn)$, to a single sum.

\begin{theorem} The congruence
\begin{align} \label{red}
\sum_{r=1}^{p-1}&
(-1)^{r}\binom{3pm+3r}{pm+r}\binom{2pm+2r}{pm+r} 
\binom{pn}{3pm+3r}\binom{p(n+m)+r}{pm+r}3^{-3r}\\\nonumber
&\equiv_{p^3}
p\binom{3m}{m}\binom{2m}{m} 
\binom{n}{3m}\binom{n+m}{m}q_p(3^{-(n-3m)})
\end{align}
or
\begin{align*} 
\sum_{r=0}^{p-1}&
(-1)^{r}\binom{3pm+3r}{pm+r}\binom{2pm+2r}{pm+r} 
\binom{pn}{3pm+3r}\binom{p(n+m)+r}{pm+r}3^{-3r}\\\nonumber
&\equiv_{p^3}
\binom{3m}{m}\binom{2m}{m} 
\binom{n}{3m}\binom{n+m}{m}3^{-(n-3m)(p-1)}
\end{align*}
implies $a_0(pn)\equiv_{p^3} a_0(n)$. 
\end{theorem}
\begin{proof}  Let $k=pm+r$ for $0\leq r< p$. Then, by using the new parameters, 
\begin{align*} 
a_0(pn)=\sum_{m=0}^{n-1}3^{p(n-3m)}(-1)^{n-m}\sum_{r=0}^{p-1}
(-1)^{r}&\binom{3pm+3r}{pm+r}\binom{2pm+2r}{pm+r} \\
&\cdot \binom{pn}{3pm+3r}\binom{p(n+m)+r}{pm+r}3^{-3r}\\
\end{align*}
Let's isolate the case $r=0$, then, from $\binom{pb}{pc}\equiv_{p^3}\binom{b}c$
and the hypothesis  we get
\begin{align*} 
a_0(pn)&\equiv_{p^3}\sum_{m=0}^{n-1}3^{p(n-3m)}(-1)^{n-m}
\binom{3m}{m}\binom{2m}{m}\binom{n}{3m}\binom{n+m}{m}\left[1+
pq_p(3^{-(n-3m)})\right]\\
&\equiv_{p^3}
\sum_{m=0}^{n-1}(-1)^{n-m}
\binom{3m}{m}\binom{2m}{m}\binom{n}{3m}\binom{n+m}{m}3^{(n-3m)}=a_0(n).
\end{align*}
\end{proof}

\section{Further Preliminary results}

\noindent
In this section, we build a few valuable results aiming at the proof of Theorem \ref{red} and hence that of Conjecture \ref{MainC}.

\begin{lemma}\label{L51} 
If $a>b\geq 0$ and $0<j<p$ then
\begin{equation}\label{L1}
\binom{ap}{bp+j} \equiv_{p^2}(a-b)\binom{a}{b}\binom{p}{j}
\quad\mbox{and}\quad
\binom{ap}{bp-j} \equiv_{p^2}b\binom{a}{b}\binom{p}{j}.
\end{equation}
Moreover, for $0\leq r<p$, 
\begin{align} \label{L2} \begin{split}
&\binom{p(n+m)+r}{pm+r} \equiv_{p^2}
\binom{n+m}{m}\left(1+n\left(\binom{p+r}{r}-1\right)\right) \end{split} \\
\label{L3}\begin{split} &\binom{2pm+2r}{pm+r} \equiv_{p^2}
\binom{2m}{m}\left( \binom{2r}{r}
      +2m\binom{p+2r}{r}-2m\binom{2r}{r}\right), \end{split} \\
\label{L4}\begin{split} 
&\binom{3pm+3r}{pm+r}\equiv_{p^2}
\binom{3m}{m}\left(2m\binom{p+3r}{r}+m\binom{p+3r}{2r}
-(3m-1)\binom{3r}{r}\right)\\
&\qquad\qquad\qquad\qquad
+\binom{3m}{m-1}\left(\binom{3r}{p+r}
+(m-1)\binom{p+3r}{2p+r}-3m\binom{3r}{p+r}\right). \end{split} \end{align}
Also, $\binom{pn}{3pm+3r}\equiv_{p^3}\frac{pn}{3pm+3r}U_r$ where
\begin{align} \label{L5}
\begin{split}
U_r\equiv_{p^2}&(3m+1)\binom{n-1}{3m+1}\left[\binom{2p-1}{3r-1}-\binom{p-1}{3r-1}
-\binom{p-1}{3r-1-p}\right]  \\
& + (3m+2)\binom{n-1}{3m+2}\left[\binom{2p-1}{3r-1-p}-\binom{p-1}{3r-1-p}
-\binom{p-1}{3r-1-2p}\right] \\
&+(3m+3)\binom{n-1}{3m+3}\left[\binom{2p-1}{3r-1-2p}-\binom{p-1}{3r-1-2p}\right]
 \\
&+3m\binom{n-1}{3m}\left[\binom{2p-1}{p+3r-1}-\binom{p-1}{3r-1}\right] \\
&+\binom{n-1}{3m}\binom{p-1}{3r-1}+\binom{n-1}{3m+1}\binom{p-1}{3r-1-p}+\binom{n-1}{3m+2}\binom{p-1}{3r-1-2p}. \end{split}
\end{align}
\end{lemma}
\begin{proof} For \eqref{L1}, we have
\begin{align*} 
\binom{ap}{bp+j}=\binom{ap}{bp}\frac{(a-b)p}{bp+j}\prod_{k=1}^{j-1}
\frac{(a-b)p-k}{bp+k}\equiv_{p^2}
(a-b)\binom{a}{b}\frac{p(-1)^{j-1}}{j}
\equiv_{p^2}
(a-b)\binom{a}{b}\binom{p}{j},
\end{align*}
and therefore
$$\binom{ap}{bp-j}=\binom{ap}{(a-b)p+j}\equiv_{p^2}
b\binom{a}{b}\binom{p}{j}.$$
For \eqref{L2}, use Vandermonde-Chu's identity and \eqref{L1} so that
\begin{align*} 
\binom{p(n+m)+r}{pm+r}&=\sum_{j=0}^r \binom{p(n+m)}{pm+j}\binom{r}{r-j}\\
&\equiv_{p^2}\binom{n+m}{m}+n\binom{n+m}{m}\sum_{j=1}^r\binom{p}{j}\binom{r}{r-j}\\
&\equiv_{p^2}\binom{n+m}{m}\left(1+n\left(\binom{p+r}{r}-1\right)\right).
\end{align*}
In a similar way, we prove \eqref{L3} as follows:
\begin{align*} 
\binom{2pm+2r}{pm+r}&=\sum_{j=-r}^r \binom{2pm}{pm+j}\binom{2r}{r-j}\\
&=\binom{2pm}{pm}\binom{2r}{r}
+\sum_{j=1}^r \binom{2pm}{pm+j}\binom{2r}{r-j}
+\sum_{j=1}^r \binom{2pm}{pm-j}\binom{2r}{r+j}\\
&\equiv_{p^2}\binom{2m}{m}\left(
\binom{2r}{r}+m\sum_{j=1}^r\binom{p}{j}\binom{2r}{r-j}
+m\sum_{j=1}^r\binom{p}{p-j}\binom{2r}{r+j}\right)\\
&\equiv_{p^2}
\binom{2m}{m}\left( \binom{2r}{r}
      +2m\left(\binom{p+2r}{r}-\binom{2r}{r}\right)\right).
\end{align*}
Moreover,
\begin{align*} 
\binom{3pm+3r}{pm+r}&=\sum_{j=-2r}^r \binom{3pm}{pm+j}\binom{3r}{r-j}\\
&=\binom{3pm}{pm}\binom{3r}{r}
+\sum_{j=1}^r \binom{3pm}{pm+j}\binom{3r}{r-j}
+\sum_{j=1}^{2r} \binom{3pm}{pm-j}\binom{3r}{r+j}\\
&\equiv_{p^2}\binom{3m}{m}\left(
\binom{3r}{r}+2m\left(\binom{p+3r}{r}-\binom{3r}{r}\right)
\right)
+\sum_{j=1}^{2r} \binom{3pm}{pm-j}\binom{3r}{r+j}.
\end{align*}
Now, \eqref{L4} is equal to
\begin{align*} 
&\sum_{j=1}^{p-1} \binom{3pm}{pm-j}\binom{3r}{r+j}
+\binom{3pm}{pm-p}\binom{3r}{r+p}+\sum_{j=p+1}^{2r} 
\binom{3pm}{pm-j}\binom{3r}{r+j}\\
&\qquad
=\sum_{j=1}^{p-1} \binom{3pm}{pm-j}\binom{3r}{r+j}
+\binom{3pm}{pm-p}\binom{3r}{r+p}
+\sum_{j=1}^{2r-p} \binom{3pm}{p(m-1)-j}\binom{3r}{r+p+j}\\
&\qquad
\equiv_{p^2}m\binom{3m}{m}\left(\binom{p+3r}{p+r}-\binom{3r}{r}-\binom{3r}{r+p}\right)
+\binom{3m}{m-1}\binom{3r}{r+p}\\
&\qquad\qquad
+(m-1)\binom{3m}{m-1}\sum_{j=1}^{2r-p} \binom{p}{p-j}\binom{3r}{r+p+j}\\
&\qquad
\equiv_{p^2}m\binom{3m}{m}\left(\binom{p+3r}{2r}-\binom{3r}{r}-\binom{3r}{p+r}\right)
+\binom{3m}{m-1}\binom{3r}{p+r}\\
&\qquad\qquad
+(m-1)\binom{3m}{m-1}\left(\binom{p+3r}{2p+r}-\binom{3r}{p+r}\right).
\end{align*}
The proof of the last congruence in \eqref{L5} is analogous and hence is omitted here.
\end{proof}

\begin{proof} We provide an alternative proof of Lemma \ref{L51} by reviving certain results found in ~\cite{BS} as equations (26) and (27), respectively. These are stated follows. If $n=n_1p+n_0$ and $k=k_1p+k_0$ where $0<n_0, k_0<p$ then
\begin{align}\label{Sagan1}\binom{np}{k}\equiv_{p^2}n\binom{n-1}{k_1}\binom{p}{k_0}, \end{align}
\begin{align}\label{Sagan2} \binom{n}{k}\equiv_{p^2}\binom{n_1}{k_1}\left[(1+n_1)\binom{n_0}{k_0}-(n_1+k_1)\binom{n_0-p}{k_0}-k_1\binom{n_0-p}{k_0+p}\right]. \end{align}
For \eqref{L1} of the lemma, apply \eqref{Sagan1} with $n_1=a, n_0=0, k_1=b, k_0=j$. So,
\begin{align*} \binom{ap}{bp+j}\equiv_{p^2}a\binom{a-1}{b}\binom{p}{j}=(a-b)\binom{a}{b}\binom{p}{j}. \end{align*}
For \eqref{L2}, apply \eqref{Sagan2} with $n_1=n+m, n_0=r=k_0, k_1=m$. So,
\begin{align*} \binom{p(n+m)+r}{pm+r}\equiv_{p^2}\binom{n+m}{m}
\left[(1+m+n)\binom{r}{r}-(n+2m)\binom{r-p}{r}-m\binom{r-p}{r+p}\right]
\end{align*}
To put this in the desired format consider applying \eqref{Sagan2} to $\binom{p+r}{r}\equiv_{p^2}2-\binom{r-p}{p}$ (with $n_1=1, n_0=k_0=r, k_1=0$); to $\binom{r-p}{r+p}=\binom{-p+r}{-2p}\equiv_{p^2}-3+2\binom{r-p}{r}$ (with $n_1=-1, n_0=r, k_1=-2, k_0=0$). After substitution and simplifications, the desired outcome is reached.

\smallskip
\noindent
For \eqref{L3}, apply \eqref{Sagan2} with $n_1=2m, n_0=2r, k_1=m, k_0=r$. So,
\begin{align*} \binom{2pm+2r}{pm+r}\equiv_{p^2}\binom{2m}{m}
\left[(1+2m)\binom{2r}{r}-3m\binom{2r-p}{r}-m\binom{2r-p}{r+p}\right]. \end{align*}
Let's reformulate this to get the result as stated in the lemma. To this end, employ \eqref{Sagan2} to $\binom{p+2r}{r}\equiv_{p^2}2\binom{2r}{r}-\binom{2r-p}{r}$ (with $n_1=1, n_0=2r, k_1=0, k_0=r$); to $\binom{p+2r}{r}=\binom{p+2r}{p+r}\equiv_{p^2}2\binom{2r}{r}-2\binom{2r-p}{r}-\binom{2r-p}{r+p}$ (with $n_1=k_1=1, n_0=2r, k_0=r$). Routine substitution completes the argument. 

\smallskip
\noindent
The congruence \eqref{L4} demands a careful analysis. The setup begins by expressing $3r=\epsilon p+d$ where $0<d<p$ and $\epsilon\in\{0,1,2\}$ which  correspond to $0<3r<p, p<3r<2p$ and $2p<3r<3p$, respectively. Here, $\epsilon=\lfloor\frac{3r}{p}\rfloor$

\smallskip
\noindent
Let $n_1=3m+\epsilon, n_0=d, k_1=m, k_0=r$ and implement \eqref{Sagan2}. So,
\begin{align*} \binom{p(3m+\epsilon)+d}{pm+r}\equiv_{p^2} \binom{3m+\epsilon}{m}
\left[(3m+\epsilon+1)\binom{d}{r}-(4m+\epsilon)\binom{d-p}{r}-m\binom{d-p}{r+p}\right]. \end{align*}
Next, engage \eqref{Sagan1} with (with $n_1=\epsilon, n_0=d, k_1=0, k_0=r$ to get
$$\binom{3r}{r}=\binom{\epsilon p+d}{r}\equiv_{p^2}(\epsilon+1)\binom{d}{r}-\epsilon\binom{d-p}{r};$$ 
with $n_1=\epsilon+1, n_0=d, k_1=0, k_0=r$ to get
$$\binom{p+3r}{r}=\binom{(\epsilon+1)p+d}{r}\equiv_{p^2}(\epsilon+2)\binom{d}{r}-(\epsilon+1)\binom{d-p}{r};$$
with $n_1=\epsilon+1, n_0=d, k_1=1, k_0=r$ to get
$$\binom{p+3r}{2r}=\binom{(\epsilon+1)p+d}{p+r}\equiv_{p^2}(\epsilon+1)(\epsilon+2)\binom{d}{r}-(\epsilon+1)(\epsilon+2)\binom{d-p}{r}
-(\epsilon+1)\binom{d-p}{r+p}.$$
After proper substitutions, the result becomes
\begin{align*} \binom{3pm+3r}{pm+r} \equiv_{p^2}&\binom{3m+\epsilon}{m}\binom{3r}{r} \\
+&\binom{3m+\epsilon}{m}\left(m\left[\frac1{\epsilon+1}\binom{p+3r}{2r}-\binom{3r}{r}\right]
+2m\left[\binom{p+3r}{r}-\binom{3r}{r}\right]\right).  \end{align*}
For \eqref{L5}, apply \eqref{Sagan1} with $n_1=n-1, n_0=p-1, k_1=3m+\epsilon, k_0=d-1$. Follow this through using $\binom{-1}{j}=(-1)^j$. The outcome is:
\begin{align} \label{L5V2} \begin{split} \binom{pn}{3pm+3r}&=\frac{pn}{3pm+3r}\binom{p(n-1)+p-1}{p(3m+\epsilon)+d-1} \\ 
&\equiv_{p^3}\frac{pn}{3pm+3r}\binom{n-1}{3m+\epsilon}\left[n\binom{p-1}{3r-1-\epsilon p}+(-1)^{r-\epsilon}(n-1)\right]. \end{split}
\end{align}
Although doable, we opt to leave this congruence in its present form instead of committing to transform it into \eqref{L5} because \eqref{L5V2} will be more convenient for our subsequent calculations.
\end{proof}

\begin{corollary}\label{C52} 
For $p>3$ a prime and an integer $0\leq r<p$, we have the congruence
\begin{align*} 
\binom{3pm+3r}{pm+r}\binom{2pm+2r}{pm+r} 
&\equiv_{p^2}\binom{3m}{m,m,m}\binom{3r}{r,r,r}\left[1+3pm(H_{3r}-H_r)\right].
\end{align*}
\end{corollary}
\begin{proof} This is a consequence of Lemma \ref{L51} and \eqref{Sagan2}. However, we offer a more direct approach.  Since $(pm+k)^{-1}\equiv_{p^2}\frac1k\left(1-\frac{pm}k\right)$, we obtain $(pm+k)^{-3}\equiv_{p^2}\frac1{k^3}\left(1-\frac{pm}k\right)^3
\equiv_{p^2}\frac1{k^3}\left(1-\frac{3pm}{k}\right)=\frac1{k^4}(k-3pm)$. For notational simplicity, denote $\binom{3j}{j,j,j}=\binom{3j}{j}\binom{2j}{j}$ by $\binom{3j}{j^3}$. We consider the expansion $\prod_{i=1}^{n}(\lambda_i+x)=\sum_{j=0}^ne_j(\lambda)x^{n-j}$ as our running theme, where $e_j$ is the $j$-th \emph{elementary symmetric function} in the parameters $\lambda=(\lambda_1,\dots,\lambda_n)$. In particular, $e_n=1$ and $e_{n-1}(1,\dots,n)=n!H_n$. The claim then follows from
\begin{align*}
\binom{3pm+3r}{(pm+r)^3}&=\binom{3pm}{(pm)^3}\prod_{j=1}^{3r}(j+3pm)\prod_{k=1}^r(pm+k)^{-3} \\
&\equiv_{p^2}\binom{3pm}{(pm)^3}\frac1{r!^4}\prod_{j=1}^{3r}(j+3pm)\prod_{k=1}^r(k-3pm) \\
&\equiv_{p^2}\binom{3pm}{(pm)^3}\frac1{r!^4}\,(3r)!r!\left[1+3pmH_{3r}-3pmH_r\right].
\end{align*}
\end{proof}

\noindent
This fact is even more general as stated below but its proof is left to the interested reader.
\begin{xca} If $A>0, 0\leq r<p$ are integers and $p>3$ a prime, then
\begin{align*} 
\binom{Apm+Ar}{pm+r,\dots,pm+r}&:=\frac{(Apm+Ar)!}{(pm+r)!^A}
\equiv_{p^2}\binom{Am}{m,\cdots,m}\binom{Ar}{r,\cdots,r}\left[1+Apm(H_{Ar}-H_r)\right].
\end{align*}
\end{xca}

\begin{corollary}\label{C54} 
For $p>3$ a prime and an integer $0\leq r<p$, we have
\begin{align*} \binom{p(n+m)+r}{pm+r}\equiv_{p^2}\binom{n+m}{m}[1+pnH_r].
\end{align*}
\end{corollary}
\begin{proof} It is easy to check that $\binom{p+r}{r}=\frac1{r!}\prod_{j=1}^r(p+j)\equiv_{p^2}1+pH_r$. The rest follows from \eqref{L2} of Lemma \ref{L51}.
\end{proof}

\begin{corollary}\label{C55} 
Let $N=n-3m$. For $p>3$ a prime and an integer $0< r<p$, it holds that
\begin{align*} \binom{pn}{3pm+3r}\equiv_{p^3}
\left(\frac{p}{3r}-\frac{p^2m}{3r^2}\right)(-1)^{r}\binom{n}{3m}\cdot 
\begin{cases} 
\displaystyle N(-1+pnH_{3r-1}), &\mbox{if } 0<r< \frac{p}3 \vspace{2mm}\\
\displaystyle\binom{N}{2}\frac{2(1-pnH_{3r-1-p})}{3m+1}, 
&\mbox{if } \frac{p}3<r< \frac{2p}3 \vspace{2mm}\\
\displaystyle \binom{N}{3}\frac{6(-1+pnH_{3r-1-2p})}{(3m+1)(3m+2)}, 
&\mbox{if } \frac{2p}3<r<p. \end{cases} 
\end{align*}
\end{corollary}
\begin{proof} We continue where we left off \eqref{L5V2} with $\epsilon=\lfloor\frac{3r}p\rfloor$. That is,
\begin{align*} \binom{pn}{3pm+3r}
&\equiv_{p^3}\frac{pn}{3pm+3r}\binom{n-1}{3m+\epsilon}\left[n\binom{p-1}{3r-1-\epsilon p}+(-1)^{r-\epsilon}(n-1)\right]. 
\end{align*}
Combining this step and the easy facts $\frac1{3pm+3r}\equiv_{p^2}\frac1{3r}-\frac{pm}{3r^2}, \binom{p-1}{j}\equiv_{p^2}(-1)^j[1-pH_j]$, we reach the desired conclusion.
\end{proof}

\begin{lemma}\label{L57} If $p>3$ is a prime then
\begin{align}\label{T1}
&\sum_{r=1}^{p-1}
\binom{3r}{r,r,r}\frac{3^{-3r}}{r}
\equiv_{p^2}
-3q_p(1/3)+\frac{3p}{2}q_p(1/3)^2, \\
\label{T2}
&\sum_{r=1}^{p-1}
\binom{3r}{r,r,r}\frac{3^{-3r}}{r^2}
\equiv_{p}
-\frac{9}{2}q_p(1/3)^2,\\
\label{T3}
&\sum_{r=1}^{p-1}
\binom{3r}{r,r,r}\frac{(H_{3r}-H_{r})3^{-3r}}{r}
\equiv_{p} 0.
\end{align}
\end{lemma}
\begin{proof} By \eqref{FT},
$q_p(1/27)\equiv_{p^2}3\,q_p(1/3)+3p\,q_p(1/3)^2$.
Therefore, by (5) in \cite[Theorem 4]{T},
\begin{align*}
\sum_{r=1}^{p-1}\binom{3r}{r,r,r}\frac{3^{-3r}}{r}
&=\sum_{r=1}^{p-1}\frac{(1/3)_r(2/3)_r}{(1)_r^2}\cdot\frac{1}{r}
\equiv_{p^2} -q_p(1/27)+\frac{p}{2}q_p(1/27)^2\\
&\equiv_{p^2} -3q_p(1/3)+\frac{3p}{2}q_p(1/3)^2.
\end{align*}
In a similar way, by (6) in \cite[Theorem 4]{T},
\begin{align*}
\sum_{r=1}^{p-1}\binom{3r}{r,r,r}\frac{3^{-3r}}{r^2}
&=\sum_{r=1}^{p-1}\frac{(1/3)_r(2/3)_r}{(1)_r^2}\cdot\frac{1}{r^2}
\equiv_{p} -\frac{1}{2}q_p(1/27)^2\equiv_{p} -\frac{9}{2}q_p(1/3)^2.\\
\end{align*}
By (1) in \cite[Theorem 1]{T},
$$\frac{(1/3)_r (2/3)_r}{(1)_r^2}
\sum_{j=0}^{r-1}\left(\frac{1}{1/3+j}+\frac{1}{2/3+j}
\right)=\sum_{k=0}^{r-1}\frac{(1/3)_k (2/3)_k}{(1)_k^2}
\cdot\frac{1}{r-k}.$$
Hence \eqref{T3} is implied by the following
\begin{align*}
\sum_{r=1}^{p-1}
\binom{3r}{r,r,r}\frac{(3H_{3r}-H_{r})3^{-3r}}{r}
&=\sum_{r=1}^{p-1} \frac{(1/3)_r (2/3)_r}{(1)_r^2}\cdot\frac{1}{r}
\cdot\sum_{j=0}^{r-1}\left(\frac{1}{1/3+j}+\frac{1}{2/3+j}\right)\\
&=\sum_{r=1}^{p-1}\frac1r
\sum_{k=0}^{r-1}\frac{(1/3)_k (2/3)_k}{(1)_k^2}\cdot\frac{1}{r-k}\\
&=
\sum_{k=0}^{p-2}\frac{(1/3)_k (2/3)_k}{(1)_k^2}
\sum_{r=k+1}^{p-1}\frac{1}{r(r-k)}\\
&=
\sum_{r=1}^{p-1}\frac{1}{r^2}+\sum_{k=1}^{p-2}\frac{(1/3)_k (2/3)_k}{(1)_k^2}
\left(\frac{1}{k}\sum_{r=k+1}^{p-1}\left(\frac{1}{r-k}-\frac{1}{r}\right)\right)\\
&\equiv_p
\sum_{k=1}^{p-2}\frac{(1/3)_k (2/3)_k}{(1)_k^2}
\cdot \frac{1}{k}\left(H_{p-1-k}-H_{p-1}+H_{k}\right)\\
&\equiv_p
\sum_{k=1}^{p-1}
\binom{3k}{k,k,k}\frac{2H_{k}\,3^{-3k}}{k},
\end{align*}
because $H_{p-1-k}\equiv_p H_k$ and $H_{p-1}\equiv_p \sum_{r=1}^{p-1}\frac{1}{r^2}\equiv_p\sum_{j=1}^{p-1}j\equiv_p 0$ as $p\neq2$.
\end{proof}

\section{Proof of Conjecture \ref{MainC}}

\noindent
In this section, we combine the results from the preceding sections to arrive at a proof for Theorem \ref{red} (restated here for the reader's convenience) and therefore for Conjecture \ref{MainC}.

\begin{theorem} For a prime $p>3$ and $m, n\in \mathbb{N}^+$, we have
\begin{align*} 
\sum_{r=1}^{p-1}&
(-1)^{r}\binom{3pm+3r}{pm+r}\binom{2pm+2r}{pm+r} 
\binom{pn}{3pm+3r}\binom{p(n+m)+r}{pm+r}3^{-3r}\\\nonumber
&\equiv_{p^3}
p\binom{3m}{m}\binom{2m}{m} 
\binom{n}{3m}\binom{n+m}{m}q_p(3^{-(n-3m)}).
\end{align*}
\end{theorem}
\begin{proof} Based on Corollaries \ref{C52}, \ref{C54}, \ref{C55} and the congruence \eqref{FT}, the assertion is equivalent to 
\begin{align}\label{closecong} 
\sum_{r=1}^{p-1}&
\binom{3r}{r,r,r} (1+3pm(H_{3r}-H_r))(1+pnH_r)
\left(\frac{1}{3r}-\frac{pm}{3r^2}\right)B_r(p,n,m)
3^{-3r}\\\nonumber
&\equiv_{p^2}
q_p(1/3)+\frac{p(N-1)}{2}q_p(1/3)^2;
\end{align}
where
\begin{align*} B_r(p,n,m)=
\begin{cases} 
\displaystyle
-1+pnH_{3r-1}, &\mbox{if }  0<r< \frac{p}3 \vspace{2mm}\\
\displaystyle
\frac{(N-1)(1-pnH_{3r-1-p})}{3m+1}, &\mbox{if }\frac{p}3<r< \frac{2p}3 \vspace{2mm}\\
\displaystyle
\frac{(N-1)(N-2)(-1+pnH_{3r-1-2p})}{(3m+1)(3m+2)}, &\mbox{if } \frac{2p}3<r<p. 
\end{cases}
\end{align*}
Now we split the sum on the left-hand side of \eqref{closecong} into three pieces according as
$$S_1=\sum_{r=1}^{\lfloor p/3\rfloor}(\cdot), \qquad
S_2=\sum_{r=\lceil p/3\rceil}^{\lfloor 2p/3\rfloor}(\cdot), \qquad\text{and} \qquad
S_3=\sum_{r=\lceil 2p/3\rceil}^{p-1}(\cdot).$$
As regards $S_1$, 
$$S_1\equiv_{p^2}\frac{1}{3}
\sum_{r=1}^{\lfloor p/3\rfloor}
\binom{3r}{r,r,r} 
\left(-\frac{1}{r}-\frac{pN}{3r^2}+\frac{pN(H_{3r}-H_r)}{r}\right)
3^{-3r}.$$
If $\frac{p}3<r<\frac{2p}3$ then $\binom{3r}{r,r,r}\equiv_p0$ and $1+3pm(H_{3r}-H_r)\equiv_{p} 1+3m$ 
with $B_r(p,n,m)\equiv_{p}\frac{(N-1)}{(3m+1)}$.
These imply that
\begin{align*}
S_2&\equiv_{p^2}
\sum_{r=\lceil p/3\rceil}^{\lfloor 2p/3\rfloor}
\binom{3r}{r,r,r} (1+3pm(H_{3r}-H_r))(1+pnH_r)
\left(\frac{1}{3r}-\frac{pm}{3r^2}\right)B_r(p,n,m)
3^{-3r}\\
&\equiv_{p^2}
\sum_{r=\lceil p/3\rceil}^{\lfloor 2p/3\rfloor}
\binom{3r}{r,r,r} (1+3m)
\left(\frac{1}{3r}\right)\frac{(N-1)}{(3m+1)}
3^{-3r}\\
&\equiv_{p^2}\frac{(N-1)}{3}\sum_{r=\lceil p/3\rceil}^{\lfloor 2p/3\rfloor}
\binom{3r}{r,r,r}\frac{3^{-3r}}{r}. 
\end{align*}
Finally, we have that $S_3\equiv_{p^2}0$ because obviously $\binom{3r}{r,r,r}\equiv_{p^2}0$
as long as $\frac{2p}3<r<p$. 

\smallskip
\noindent 
Again $\binom{3r}{r,r,r}\equiv_p0$ if $\frac{p}3<r<\frac{2p}3$ and $\binom{3r}{r,r,r}\equiv_{p^2}0$ if $\frac{2p}3<r<p$. So, from Lemma \ref{L57} we know
\begin{align*}
\sum_{r=1}^{\lfloor 2p/3\rfloor}\binom{3r}{r,r,r}\frac{3^{-3r}}{r}
&\equiv_{p^2}\sum_{r=1}^{p-1}
\binom{3r}{r,r,r}\frac{3^{-3r}}{r}
\equiv_{p^2}
-3q_p(1/3)+\frac{3p}{2}q_p(1/3)^2, \\
p\sum_{r=1}^{\lfloor p/3\rfloor}\binom{3r}{r,r,r}\frac{3^{-3r}}{r^2}
&\equiv_{p^2} p\sum_{r=1}^{p-1}
\binom{3r}{r,r,r}\frac{3^{-3r}}{r^2}
\equiv_{p^2} -\frac{9p}{2}q_p(1/3)^2.
\end{align*}
As before $\binom{3r}{r,r,r}\equiv_{p^2}0$ for $\frac{2p}3<r<p$.  As well as $\binom{3r}{r,r,r}\equiv_p0$ and $pH_{3r}-pH_r\equiv_p 1$
for $\frac{p}3<r<\frac{2p}3$. Therefore, by Lemma \ref{L57}
\begin{align*}
0\equiv&_{p^2} \,
p\sum_{r=1}^{p-1}\binom{3r}{r,r,r}\frac{(H_{3r}-H_r)3^{-3r}}{r}
\equiv_{p^2}
p\sum_{r=1}^{\lfloor 2p/3\rfloor}
\binom{3r}{r,r,r}\frac{(H_{3r}-H_r)3^{-3r}}{r}\\
&\equiv_{p^2} p\sum_{r=1}^{\lfloor p/3\rfloor}
\binom{3r}{r,r,r}\frac{(H_{3r}-H_r)3^{-3r}}{r} 
 +\sum_{r=\lceil p/3\rceil}^{\lfloor 2p/3\rfloor}
\binom{3r}{r,r,r}\frac{3^{-3r}}{r}.
\end{align*}
Putting all these together, we conclude that
\begin{align*}
S_1+S_2+S_3&\equiv_{p^2} 
\frac{1}{3}
\sum_{r=1}^{\lfloor p/3\rfloor}
\binom{3r}{r,r,r} 
\left(-\frac{1}{r}-\frac{pN}{3r^2}+\frac{pN(H_{3r}-H_r)}{r}\right)
3^{-3r}\\
&\qquad\quad +\frac{(N-1)}{3}\sum_{r=\lceil p/3\rceil}^{\lfloor 2p/3\rfloor}
\binom{3r}{r,r,r}\frac{3^{-3r}}{r}+0\\
&\equiv_{p^2}-\frac13\sum_{r=1}^{\lfloor 2p/3\rfloor}\binom{3r}{r,r,r}\frac{3^{-3r}}{r} 
-\frac{N}9\sum_{r=1}^{\lfloor p/3\rfloor}\binom{3r}{r,r,r}\frac{3^{-3r}}{r^2}
+\frac{N}{3}\cdot 0 \\
&\equiv_{p^2}
-\frac{1}{3}\left(-3q_p(1/3)+\frac{3p}{2}q_p(1/3)^2\right)
-\frac{N}{9}\left(-\frac{9p}{2}q_p(1/3)^2\right)\\
&\equiv_{p^2}
q_p(1/3)+\frac{p(N-1)}{2}q_p(1/3)^2,
\end{align*}
which is exactly what we expect. The proof is complete.
\end{proof}

\section{Conclusions and Remarks}

\noindent
In this final section, we extend the congruence on $a_i(n)$ (for $i>0$),  discussed in the earlier sections, from modulo $p^2$ to modulo $p^3$. While stating our claim in its generality, we only exhibit proof outlines for the case $i=1$ as a prototypical example. We believe the curious researcher would be able to account for the remaining cases.

\begin{conjecture} For $n,i\in\mathbb{N}^+$ and a prime $p>2i$,
$$
a_i(pn)\equiv_{p^3} (-1)^{i-1}\frac{a_1(pn)}{i^2\binom{2i-1}{i-1}}
\equiv_{p^3} \frac{(-1)^{i-1}p^2\binom{n+2}{2}a_1(n)}{i^2\binom{2i-1}{i-1}} .
$$
\end{conjecture}

\begin{proof} Ingredients for $a_1(pn)\equiv_{p^3}p^2\binom{n+2}2a_1(n)$. 

\smallskip
\noindent
(A) By partial fraction decomposition
\begin{align*}
a_i(n)&=\frac{1}{3^i}
\sum_{k=0}^{n-1}(-1)^{n-k}\binom{3k}{k}\binom{2k}{k}
             \binom{n}{3k}\binom{n+k}{k}\frac{\binom{n-3k}{i}3^{n-3k}}{\binom{k+i}{i}}\\
&=(-1)^ia_0(n)+\frac{i}{3^i}\sum_{j=1}^i(-1)^{j-1}\binom{i-1}{j-1}\binom{n+3j}{i}b_j(n)
\end{align*}
where for $j\in\mathbb{N}^+$,
$$b_j(n):=\sum_{k=0}^{n-1}(-1)^{n-k}(n-3k)\binom{3k}{k}\binom{2k}{k}
             \binom{n}{3k}\binom{n+k}{k}\frac{3^{n-3k}}{k+j}.$$
Thus, $a_0(np)\equiv_{p^3} a_0(n)$ implies
\begin{align*}
a_1(np)&=-a_0(np)+\frac{np+3}{3}\,b_1(np)\equiv_{p^3} -a_0(n)+\frac{np+3}{3}\,b_1(np)\\
&\equiv_{p^3} a_1(n)+\frac{np+3}{3}\,b_1(np)-\frac{p+3}{3}\,b_1(p).
\end{align*}
(B) Hence, it suffices to show that
$$b_1(np)\equiv_{p^3}
\frac{3}{np+3}\left(p^2\binom{n+2}{2}-1\right)a_1(n)
+\frac{n+3}{np+3}\,b_1(n),$$
or, since $a_1(n)=-a_0(n)+(n+3)b_1(n)/3$,
\begin{equation}\label{b1}
b_1(np)\equiv_{p^3}
p^2\binom{n+3}{3}b_1(n)+\left(1-\frac{pn}{3}
             -\frac{p^2(n+3)(7n+6)}{18}\right)a_0(n).
\end{equation}    
(C) The above congruence is implied by the following
\begin{align}\label{C}
\sum_{r=0}^{p-1}&
(-1)^{r}\binom{3pm+3r}{pm+r}\binom{2pm+2r}{pm+r}
\binom{pn}{3pm+3r}\binom{p(n+m)+r}{pm+r}\frac{3^{-3r}}{pm+r+1}\\\nonumber
&\equiv_{p^3}
\left(\frac{p^2}{m+1}\binom{n+3}{3}+1-\frac{pn}{3}
             -\frac{p^2(n+3)(7n+6)}{18}\right)\\\nonumber
&\qquad\qquad\qquad\qquad\cdot\binom{3m}{m}\binom{2m}{m}
\binom{n}{3m}\binom{n+m}{m}3^{-N(p-1)}
\end{align}
By summing over $m$,  it is immediate to recover \eqref{b1}. 
\smallskip

\noindent
(D) In order to prove \eqref{C}, we have the old machinery,
$\frac{1}{pm+r+1}\equiv_{p^2}
\frac{1}{r+1} -\frac{mp}{(r+1)^2},$
and 
\begin{align*}
&\sum_{r=0}^{p-1}\binom{3r}{r,r,r}\frac{3^{-3r}}{r+1}=
\frac{9p}{2}\binom{3p}{p,p,p} 3^{-3p}\equiv_{p^2}\,\, p-3p^2q_p(1/3),\\
&\sum_{r=0}^{p-1}\binom{3r}{r,r,r}\frac{3^{-3r}}{(r+1)^2}=
\frac{9(9p+2)}{4}\binom{3p}{p,p,p} 3^{-3p}-\frac{9}{2}\equiv_{p}-\frac{7}2.\\
\end{align*}
(E) Finally, we can modify a previous proof as follows:
\begin{align*}
\sum_{r=0}^{p-1}
\binom{3r}{r,r,r}\frac{(3H_{3r}-H_{r})3^{-3r}}{r+1}
&=\sum_{r=1}^{p-1} \frac{(1/3)_r (2/3)_r}{(1)_r^2}\cdot\frac{1}{r+1}
\cdot\sum_{j=0}^{r-1}\left(\frac{1}{1/3+j}+\frac{1}{2/3+j}\right)\\
&=\sum_{r=1}^{p-1}\frac{1}{r+1}
\sum_{k=0}^{r-1}\frac{(1/3)_k (2/3)_k}{(1)_k^2}\cdot\frac{1}{r-k}\\
&=
\sum_{k=0}^{p-2}\frac{(1/3)_k (2/3)_k}{(1)_k^2}
\sum_{r=k+1}^{p-1}\frac{1}{(r+1)(r-k)}\\
&=
\sum_{k=0}^{p-2}\frac{(1/3)_k (2/3)_k}{(1)_k^2}
\left(\frac{1}{k+1}\sum_{r=k+1}^{p-1}\left(\frac{1}{r-k}-\frac{1}{r+1}\right)\right)\\
&=
\sum_{k=0}^{p-2}\frac{(1/3)_k (2/3)_k}{(1)_k^2}
\cdot \frac{H_{p-1-k}-H_{p}+H_{k+1}}{k+1}\\
&\equiv_p
\sum_{k=0}^{p-1}
\binom{3k}{k,k,k}\frac{(H_{k}-H_{p }+H_{k+1})\,3^{-3k}}{k+1},
\end{align*}
which implies that
\begin{align*}
\sum_{r=0}^{p-1}
\binom{3r}{r,r,r}\frac{(H_{3r}-H_{r})3^{-3r}}{r+1}
&\equiv_p \frac{1}{3}
\sum_{k=0}^{p-1}
\binom{3k}{k,k,k}\frac{(-1/p+1/(k+1))\,3^{-3k}}{k+1}\\
&\equiv_p \frac{1}{3}\left(-1-\frac{7}{2}\right)=-\frac{3}{2}.
\end{align*}
\end{proof}

\begin{remark} We showed that the conjecture $a_0(pn)\equiv_{p^3}a_0(n)$ holds true. Although it is not pursued here, the techniques established in this paper if combined with existing literature on supercongruences (see references below) for binomials of the type $\binom{p^rn+k}{p^tm+j}$, there is enough reliable verity to believe that $a_0(p^rn)\equiv_{p^{3r}}a_0(p^{r-1}n)$ should be within easy grasp.

\end{remark}
\bigskip

\noindent
\textbf{Acknowledgements.} The first-named author is grateful to A. Straub for bringing the conjecture on the Almkvist-Zudilin numbers to his attention.

\end{document}